\documentclass[12pt,a4paper]{amsart}
\pdfoutput=1

\usepackage{amsmath}
\usepackage{amsfonts}
\usepackage{amssymb}

\makeatletter
\@namedef{subjclassname@2020}{\textup{2020} Mathematics Subject Classification}
\makeatother
\def\Z{{\Bbb Z}}
\def\Q{{\Bbb Q}}    
\def\R{{\Bbb R}}
\def\C{{\Bbb C}}

\def\fb{{\mathfrak b}}
\def\fa{{\mathfrak a}}
\def\fp{{\mathfrak p}}
\def\fc{{\mathfrak c}}
\DeclareMathOperator{\ord}{ord}

\newtheorem{theorem}{Theorem}[section]
\newtheorem{lemma}[theorem]{Lemma}
\newtheorem{prop}[theorem]{Proposition}

\newtheorem{definition}[theorem]{Definition}

\newtheorem{remark}[theorem]{Remark}

\begin{document}
\title[genus character $L$-functions]
{Genus character $L$-functions of 
quadratic orders in an adelic way and maximal orders of 
matrix algebras}
\author{Tomoyoshi Ibukiyama}
\address{Department of Mathematics, Graduate School of Mathematics, Osaka University,
Machikaneyama 1-1, Toyonaka,Osaka, 560-0043 Japan}
\email{ibukiyam@math.sci.osaka-u.ac.jp}
\thanks{This work was supported by JSPS Kakenhi Grant Number JP19K03424 
and JP	20H00115}
\keywords{$L$-functions, genus theory, quadratic field, }
\subjclass[2020]{11R11, 11R37, 11R42}

\maketitle

\begin{abstract}
For a quadratic extension $K$ of $\Q$, we consider orders $O$ in $K$ 
that are not necessarily maximal and 
the ideal class group 
$Cl^+(O)$ in the narrow sense of proper ideals of $O$. 
Characters of $Cl^+(O)$  of order at most two are 
traditionally called genus characters. Explicit description of such characters 
is known classically, but explicit $L$-functions associated 
to those characters are only recently obtained partially by Chinta and Offen and completely by Kaneko and Mizuno.
As remarked in the latter paper, the present author 
also obtained the formula of such $L$-functions independently.
Indeed, here we will give a simple and transparent alternative proof of the formula 
by rewriting explicit genus characters and their values   
in an adelic way starting from scratch. 
We also add 
an explicit formula for the genus number in the wide sense, 
which is maybe known but rarely treated. As an appendix 
we give an ideal-theoretic characterization of 
isomorphism classes of maximal orders of the matrix algebras $M_n(F)$ 
over a number field $F$ up to 
$GL_n(F)$ and $GL_n^+(F)$ conjugation respectively, and 
apply genus numbers to count them 
when $n=2$ and $F$ is quadratic.
To avoid any misconception, we include some easy known details. 
\end{abstract}

\section{Introduction}
The purpose of the paper is to 
give an alternative proof of the formula for the genus 
character $L$-functions associated with 
not necessarily maximal orders of quadratic fields
given in \cite{chinta} except for some cases and 
in \cite{kanekomizuno} for all the cases.
For that purpose, we describe proper ideals 
of non-maximal quadratic orders, 
their genus characters, and their values at ideals 
all explicitly in an adelic way.
A global description of such objects is 
a classical result (see for example
\cite{dirichlet} or \cite{weber}), but results on 
the $L$-function are obtained recently (see \cite{chinta}, \cite{kanekomizuno}).
Anyway, this paper is more or less 
expository in nature, and our new point here is to 
treat everything adelically. 
Almost from scratch except for an easy part of 
the class field theory, we give explicit formulas of 
genus characters and $L$-functions associated with it.
This would give more transparent view to the theory. 
When the order is not maximal, we do not know 
a reference treating this subject in this way, so 
we believe it is not useless to publish this. 
We also add a formula for the genus number in 
the wide sense, which has application to conjugacy classes 
of maximal orders of $2\times 2$ matrix algebras.
Indeed in the final section, for general $n$ and algebraic number fields $F$, 
we consider ideal theoretic 
characterization of the number of maximal orders of $M_n(F)$ 
up to $GL_n(F)$ conjugation and $GL_n^+(F)$ conjugation, where
$GL_n^+(F)$ means those with totally positive determinant.
The result for $GL_n(F)$ conjugation has been known in \cite{arima}.

The paper is outlined as follows.
In the next section, we review 
the theory of genus of cyclic extensions $K$ of 
$\Q$ for orders $O$ of $K$ not necessarily maximal.
This is a minor generalization of \cite{iyanagatamagawa}, where the 
case of maximal orders is treated.
In section $3$, we assume that $K$ is quadratic, and 
explicitly describe proper ideals of non-maximal orders, 
the adelic subgroup corresponding to the genus, and genus characters. 
Then give formulas of values of genus characters at ideals. 
In section 4, we give an explicit formula of $L$-functions associated to  genus characters (see Theorem \ref{main}). 
In section 5, we give a formula for the genus number in the wide sense. 
In section 6, we give a general theory on the number of maximal orders
of the matrix algebras over an algebraic number field, and 
for a quadratic field $K$, we give an application of the genus numbers 
in the wide sense and in the narrow sense 
to the number of $GL_2(K)$ and $GL_2^+(K)$ conjugacy classes of 
maximal orders in $M_2(K)$. 

In this paper, we do not explain the relation to quadratic forms.
For this important part of the history, 
see \cite{dirichlet}, \cite{weber}, 
\cite{cassels} or \cite{arakawaibukaneko} for example. 
In fact, this paper would be read as an appendix to (the Japanese version of)
\cite{arakawaibukaneko}, where everything was treated globally. 
For some old history of the genus theory, see \cite{lemmermeyer}, 
and further generalization of the notion of the 
genus, see \cite{mhorie} and the references there.

\section{Definition of a genus for cyclic extensions 
over $\Q$}
Our main concern is a quadratic order, but in this section we review 
the genus theory of cyclic extensions of $\Q$ based on 
\cite{iyanagatamagawa}, since it would make 
our points clearer.
Here the only difference from \cite{iyanagatamagawa}
is that we describe the theory for 
orders $O$ not necessarily maximal.
For reader's convenience, we repeat some arguments there.
 
Let $V$ be any finite dimensional vector space over 
$\Q$. 
A free $\Z$ submodule $L$ of $V$ is said to be  
a lattice if it contains a basis of $V$ over $\Q$.
When $K$ is an algebraic number field, 
regarding $K$ as a vector space over $\Q$, 
a subring $O$ of $K$ that contains $1$ and is a lattice 
of $K$  is called an order of $K$.
It is clear that any element of $O$ is an algebraic 
integer, so $O\subset O_{max}$, where $O_{max}$ is 
the maximal order of $K$. 
While $O_{max}$ is a Dedekind domain, the order 
$O\subsetneqq O_{max}$ is not Dedekind since 
it is not integrally closed. So there is no 
prime ideal decomposition of ideals of $O$. 
Besides, for an ideal $\fa$ of $O$, there is 
no inverse ideal in general.  This means that if we want to 
define ideal classes, we must restrict ideals 
to a smaller set of ideals of $O$. According to the usual habit,  
a lattice $L$ of $K$ with $OL\subset L$ is 
called a fractional ideal of $O$. If $L\subset O$ besides,
we say that $L$ is an integral ideal, or just an ideal 
of $O$.
We write $V_p=V\otimes_{\Z}\Q_p$ for any prime $p$ where $\Q_p$ is 
the field of $p$-adic numbers.
For any submodule $L$ of $V$ and a prime 
$p$, we write $L_p=L\otimes_{\Z}\Z_p\subset V_p$, where 
$\Z_p$ is the ring of $p$-adic integers.

\begin{definition}
We say that a fractional ideal $\fa$ of $O$ is locally principal 
if $\fa_p=O_pa_p$ for some $a_p\in K_p$ for 
every prime $p$.
\end{definition}

Here we note that $K_p=K\otimes_{q}\Q_p$ is a direct sum of 
fields according to the decomposition of $p$ in $K$ 
and not a field in general. The ring $O_p$ is not necessarily
decomposed into a direct sum of orders of the fields, 
and it is not suitable to consider each place of $K$ over $p$ separately
when $O$ is not maximal.
The relation between $\fa$ and the collection of $\fa_p$
for any $p$ is given by the proposition given below.
Any $\Z_p$ submodule $L$ of $V_p$ 
is called a lattice of $V_p$ if 
$L=\Z_p\omega_1+\cdots+\Z_p\omega_n$ for some 
basis $\{\omega_1,\ldots,\omega_n\}$ of $V_p$ over 
$\Q_p$.

\begin{prop}
Notation being as above, 
let $\{N_p\}_{p:prime}$ be a family of lattices in $V_p$
and $L$ be a lattice in $V$. 
Assume that $L_p=N_p$ for almost all $p$. Then there 
exists a lattice $M$ in $V$ such that $N_p=M\otimes_{\Z}\Z_p$ and 
\[
M=\bigcap_{p:prime}(V\cap N_p).
\]
\end{prop}

For the proof, 
see Weil \cite{weilbook} p.84 Theorem 2. 

We denote by $K_A^{\times}$ the group of ideles of $K$.
For any element $a=(a_v)\in K_A^{\times}$, 
we may define a locally principal fractional 
ideal $\fa$ of $O$ by
\[
\fa=\bigcap_{p:prime}(a_pO_p\cap K), \qquad (a_p=(a_v)_{v|p}\in K_p).
\]
So locally principal fractional ideals correspond
with $K_{A,fin}^{\times}/\prod_{p}O_p^{\times}$, where
$K_{A,fin}^{\times}$ is the finite part of the ideles, i.e. the group of 
ideles whose components at infinite places are all $1$. 
For a locally principal fractional ideal $\fa$ of $O$ as above, we may 
define an inverse ideal by 
\[
\fa^{-1}=\bigcap_{p:prime}(a_p^{-1}O_p\cap K).
\]
Then we have $\fa\fa^{-1}=O$ and 
locally principal fractional ideals of $O$ form a group.
We say that locally principal 
fractional ideals $\fa$ and $\fb$ are 
equivalent in the wide sense if $\fb=\fa\alpha$ for some 
$\alpha\in K^{\times}$. Equivalence in the narrow sense is 
defined by imposing a condition that $\alpha\in 
K_+^{\times}$, where $K_+$ is the set of totally positive 
elements $\alpha$ of $K$, that is, $\alpha$ is 
positive under embeddings of $K$ into the real field at all infinite 
real places and no condition at complex places.
We denote by $Cl(O)$ (resp. $Cl^+(O)$) the group of 
classes of locally principal fractional ideals in the wide sense 
(resp. in the narrow sense).
We will mainly consider $Cl^+(O)$.
As usual, $K^{\times}$ is diagonally embedded in 
$K_A^{\times}$, and we denote the image by the same 
letter $K^{\times}$.
Let $r_1$ and $r_2$ be the number of real places and 
complex places of $K$, respectively.
We put 
\[
U_{\infty}=(\R^{\times})^{r_1}\times (C^{\times})^{r_2}
\text{ and }
U_{\infty,+}=(\R_+^{\times})^{r_1}\times (\C^{\times})^{r_2},
\]
where $\R_+^{\times}$ is the set of positive real 
numbers. 
(If $K/\Q$ is Galois, we have $r_2=0$ if $K\subset \R$ 
and $r_1=0$ if not.)
We put $U(O)=U_{\infty}\prod_{p:prime}O_p^{\times}$ and 
$U_+(O)=U_{\infty,+}\prod_{p:prime}O_p^{\times}$.
Then we have 
\begin{align*}
Cl(O) & \cong K_A^{\times}/K^{\times}U(O),\\
Cl^+(O) & \cong K_A^{\times}/K^{\times}U_{+}(O)
\cong U_{\infty,+}K_{A,fin}/K_+^{\times}U_+(O).
\end{align*}
The last isomorphism comes from the fact that $K$
contains elements of any signature at infinite places.

Here we shortly review the class field theory over $\Q$.
\begin{lemma}[Class field theory]
The set of abelian extensions $K$ of $\Q$ 
in the algebraic closure of $\Q$ is bijective 
to the set of finite index subgroups $H$ of $\Q_A^{\times}$ containing
$\Q^{\times}$. Here, denoting by $N_{K/\Q}$ the norm from $K$ to $\Q$ and 
by $Gal(K/\Q)$ the Galois group of $K$ over $\Q$,  
the correspondence is given by 
\[
H=\Q^{\times}N_{K/\Q}(K_A^{\times}), \qquad Gal(K/\Q)\cong \Q_A^{\times}/H.
\] 
\end{lemma}
The following direct product decomposition is well known and easy to see.
\begin{equation}\label{direct}
\Q_A^{\times}=\Q^{\times}\times \R_+^{\times}
\times \prod_{p:prime}\Z_p^{\times}.
\end{equation}
The Galois group of the maximal abelian extension of $\Q$
(that is, the union of all cyclotomic fields)
is given by $\prod_{p}\Z_p^{\times}=\varprojlim\limits_{N}
(\Z/N\Z)^{\times}$.
The relation of this fact to the class field theory is 
as follows. By the direct product \eqref{direct}, we see that 
any $H$ in the lemma can be written as  
\[
H=\Q^{\times}\times (\R_+^{\times}\times H_0),  
\qquad H_0\subset 
\prod_{p:prime}\Z_p^{\times}.
\]
Then we see that  
\[
\Q_A^{\times}/H\cong \left(\prod_{p}\Z_p^{\times}\right)/H_0.
\]
Of course $H_0$ is in general bigger than
$\prod_{p:prime}(H_0\cap \Z_p^{\times})$, 
where $\Z_p^{\times}$ is identified with 
the subset of $\Q_A^{\times}$ whose components 
at places $v\neq p$ are all $1$ while components at $p$ are in $\Z_p^{\times}$.
If we write 
\[
e_p=[\Z_p^{\times};\Z_p^{\times}\cap H_0],
\]
then $e_p$ is the ramification index of $p$ in $K$.
Indeed, for $a \in \Q^{\times}N_{K/\Q}K_A^{\times}$ written as 
$a=cu_{\infty}u_0$ with $c\in \Q^{\times}$, $u_{\infty}\in \R_+^{\times}$
and $u_0=(u_{0,q})\in H_0$, assume that $u_{0,q}=1$ unless $q\neq p$.
Denote by $\theta$ the reciprocity map of $\Q_A^{\times}$ to 
$Gal(K/\Q)$, Then we have 
$\theta(b)=\prod_{v}\theta_v(b_v)$ for any $b=(b_v)\in \Q^{\times}$,
where $\theta_v$ is the reciprocity map from $\Q_v^{\times}$ to
$Gal(K_w/\Q_v)$ for any place $w$ of $K$ over a place $v$. 
So we have $\theta(u_0)=\theta_p(u_{0,p})$.
But since $\theta(a)=\theta(c)=\theta(u_{\infty})=1$ by definition,
we have $\theta_p(u_{0,p})=1$. This means $u_{0,p}\in N_{K_{\fp}/\Q_p}(K_v^{\times})$
where $\fp$ is any prime of $K$ over $p$.
So $e_p=[\Z_p^{\times}:\Z_p^{\times}\cap N_{K_{\fp}/\Q_p}(K_{\fp}^{\times})]$. 
Here $e_p=1$ for almost all $p$.

In particular, if $K$ is cyclic over $\Q$, 
then $[K:\Q]$ is the least common multiple 
$n$ of $e_p$ defined above. 
Indeed, take a character $\chi_K$ of $\prod_q\Z_q^{\times}$ 
such that $Ker(\chi_K)=H_0$.
If we decompose $\chi_K$ as 
$\chi_K=\prod_p\chi_{K,p}$ by characters 
$\chi_{K,p}$ on $\Z_q^{\times}/(\Z_q^{\times}\cap H_0)$,
then $\chi_{K,p}$ is of order $e_p$, and the order of 
$\chi_K$ is $n$.  

From here until the end of this section, 
we assume that $K$ is a cyclic extension of $\Q$.
We fix a generator $\sigma$ of $Gal(K/\Q)$. 
For an order $O$ of $K$ which is not necessarily maximal, we put 
\[
U_+(O)=U_{\infty,+}\prod_{p}O_p^{\times}.
\]
To define a genus of $O$, we prepare the 
following proposition.

\begin{prop}\label{iyanagatamagawa}
Notation being as above, for $a \in K_A^{\times}$, 
the following conditions (1) and (2) are equivalent.
\\
(1) $N_{K/\Q}(a)\in \Q^{\times}N_{K/\Q}(U_+(O))$. \\
(2) There exists $b\in K_A^{\times}$, $u\in U_+(O)$ and 
$c\in K^{\times}$ such that 
\[
a=b^{1-\sigma}uc.
\]
\end{prop}

\begin{proof}
This is essentially Theorem 3 in \cite{iyanagatamagawa}
except for the point that we do not assume that $O$ is 
maximal. It is trivial that (2) implies (1). So we prove (2) 
assuming (1). First of all, we give an idelic version of 
Hilbert Satz 90 stated as follows: 

For any $a \in K_A^{\times}$
with $N_{K/\Q}(a)=1$, 
there exists $b\in K_A^{\times}$ such that 
$a=b^{1-\sigma}$. 

This is claimed in \cite{iyanagatamagawa} without proof, so we give 
here a proof. We have 
$K_p=F\oplus \cdots \oplus F$ 
for some field $F$ over $\Q_p$ 
(isomorphic to the completion of $K$ at any place of $K$ over $p$). 
Let $\fp$ be a prime ideal in $K$ over $p$ and 
$\tau=\sigma^m$ a generator the decomposition group of $\fp$.
Each component of $K_p$ corresponds with 
the embedding associated to $\fp^{\sigma^i}$ with some  
$i\in \{0, \ldots, m-1\}$. 
For $x=(x_1,\ldots,x_m)\in F^m=K_p$, 
we may regard 
\[
x^\sigma= (x_m^{\tau},x_1,x_2\ldots,x_{m-1}).
\]
So we have 
\[
x^{1-\sigma}=(x_1/x_m^{\tau},x_2/x_1,\ldots,x_{m}/x_{m-1}).
\]
For $y=(y_1,\ldots,y_m) \in F^m=K_p$, we have 
\[
N_{K/\Q}(y)=N_{F/\Q}(y_1\cdots y_m).
\]
The condition that 
\[
y=x^{1-\sigma}
\]
is 
\[
y_1=x_1/x_m^{\tau}, \quad y_2=x_2/x_1,\quad \ldots,\quad y_{m}=x_{m}/x_{m-1}.
\]
so $y_1\cdots y_m=x_{m}^{1-\tau}$.
Since we assumed 
$N_{K/\Q}(y)=N_{F/\Q}(y_1\cdots y_m)=1$, 
there exists $x_0\in F^{\times}$
such that $y_1\cdots y_m=x_0^{1-\tau}$ by the usual Hilbert Satz 90
for cyclic extensions. 
If we put $x_m=x_0$ and define $x_i$ inductively by 
$x_{i}=x_{i+1}y_{i+1}^{-1}$ for any $1\leq i\leq m-1$, then 
for $x=(x_1,\ldots,x_m)$, we have 
$y=x^{1-\sigma}$. Now we must show that 
$x$ is in $K_A^{\times}$. For almost all primes $p$, 
we have $O_p=O_{max,p}$ and $p$ is unramified in $K$. 
For such $p$, any element of the maximal order $O_F$ of $F$ is written as
$x_0=p^e\epsilon$ for some $\epsilon\in O_F^{\times}$.
Since $x_0^{1-\tau}=\epsilon^{1-\tau}$, we may take 
$x_0=\epsilon$. By definition of ideles, we have 
$y_i\in O_F^{\times}$
for almost all $p$, so $x_i=x_{i+1}y_{i+1}^{-1}$ is 
also in $O_F^{\times}$. So 
for almost all $p$, we may assume $y=x^{1-\sigma}$
for $x \in O_p^{\times}$. This means that 
for any $a \in K_A^{\times}$ with $N_{K/\Q}(a)=1$, 
we have $a=b^{1-\sigma}$ for some 
$b\in K_A^{\times}$. So the idelic version of 
Satz 90 is proved. Now assume (1) in Proposition 
for $a \in K_A^{\times}$. 
Then we have $N_{K/\Q}(au^{-1})\in \Q^{\times}$
for some $u \in U_+(O)$. For a cyclic extension, by the 
Hasse norm theorem, an element of 
$\Q$ is a local norm if and only if it is a global norm
(e.g. \cite{iyanagatamagawa} quoted \cite{chevalley}),
so we have $c \in K^{\times}$ such that 
$N_{K/\Q}(au^{-1}c^{-1})=1$. So 
we have $au^{-1}c^{-1}=b^{1-\sigma}$ for some $b\in 
K_A^{\times}$. 
\end{proof}

\begin{definition}
The subgroup of $H(O)$ 
of elements of $K_A^{\times}$ 
that satisfy (1) and (2) is said to be 
a principal genus of $O$. 
A coset in $K_A^{\times}/H(O)$ 
is called a genus of $O$.
We call the number of these cosets a genus number
in the narrow sense.
\end{definition}

More classical explanation is given as follows.
As we have explained, we have 
\[
Cl^+(O)\cong K_A^{\times}/K^{\times}U_+(O)\cong U_{\infty,+}K_{A,fin}^{\times}/
K_+^{\times}U_+(O).
\]
Here by definition we have
\[
K^{\times}U_+(O) \subset H(O)\subset K_A^{\times}, 
\]
so $H(O)/K^{\times}U_+(O)$ is a subgroup of 
$Cl^+(O)\cong K_A^{\times}/K^{\times}U_+(O)$. 
Elements in this subgroup are called 
the ^^ ^^ principal genus classes" in the narrow 
sense and a genus is a coset of ideal classes in the narrow sense
divided by these classes.
(When $K$ is quadratic, obviously 
the principal genus classes consists of square classes
by the condition (2) above.
The purpose of Proposition \ref{iyanagatamagawa} is 
to compare the condition (1) with 
the classical setting. For non-cyclic extensions, 
only the condition (1) is often used 
for the definition of the principal genus classes.
See for example \cite{mhorie}.)
 A character of
the group $K_A^{\times}/K^{\times}U_+(O)\cong Cl^+(O)$ which is 
trivial on $H(O)/K^{\times}U_+(O)$ is called 
a genus character.

If we consider the map 
\[
K_A^{\times}\stackrel{N_{K/\Q}}{\longrightarrow}
\Q_A^{\times}\longrightarrow \Q_A^{\times}/\Q^{\times},
\]
 then since $K/\Q$ is cyclic, the kernel is $K^{\times}$. 
So we see that 
\[
K_A^{\times}/H(O)\cong (K_A^{\times}/K^{\times})/(H(O)/K^{\times})
\cong \Q^{\times}N_{K/\Q}
(K_A^{\times})/\Q^{\times}N_{K/\Q}(U_+(O)).
\]
So the genus number $g$ of $O$ in the narrow sense 
is given by 
\begin{align*}
g& =[K_A^{\times}:H(O)]
=[\Q^{\times}N_{K/\Q}(K_A^{\times}):\Q^{\times}N_{K/\Q}
(U_+(O))]
\\ &  
=[\Q_A^{\times}:\Q^{\times}N_{K/\Q}(U_+(O)]/[\Q_A^{\times}:\Q^{\times}N_{K/\Q}(K_A^{\times})].
\end{align*}
By the class field theory we have 
\[
[\Q_A^{\times};\Q^{\times}N_{K/A}(K_A^{\times})]
=[K:\Q].
\]
On the other hand, we have 
\[
[\Q_A^{\times}:\Q^{\times}N_{K/\Q}(U_+(O))]
=
\prod_p[\Z_p^{\times}:N_{K/\Q}(O_p^{\times})].
\]
So writing $e_p=[\Z_p^{\times}:N_{K/\Q}(O_p^{\times})]$, 
we have 
\[
g=\left(\prod_{p}e_p\right)/[K:\Q].
\]
Here since $O$ might not be maximal, $e_p$ might not be the ramification 
index of $K/\Q$. The abelian extension of $\Q$ corresponding to 
$\Q^{\times}N_{K/\Q}(U_+(O))$ by the class field theory is called the 
genus field of $H(O)$.

We denote by $X(O)$ the set of characters $\phi$ 
of $\prod_p\Z_p^{\times}$ such 
that the $p$ component $\phi_p$ is a character of 
$\Z_p^{\times}/N_{K/\Q}(O_p^{\times})$. 
We can naturally prolong $\phi$ to the character of $\Q_A^{\times}$ 
by setting so that it is trivial on $\Q^{\times}\times \R_+^{\times}$. 
Now we denote by $\chi_K$ one of the characters of 
$\Q_A^{\times}$ trivial on $\Q^{\times}N_{K/\Q}(K_A^{\times})$. This is called a character  
corresponding to $K/\Q$. Of course this is trivial on
$\Q^{\times}\R_+^{\times}$, so it can be regarded 
as a character of $\prod_{p}\Z_p^{\times}$. 
Then the $p$ component $\chi_{K,p}$ of $\chi_K$ 
on $\Z_p^{\times}$ 
is a character of 
$\Z_p^{\times}/N_{K/\Q}(O_{max,p}^{\times})$. Since $O_p^{\times}
\subset O_{max}^{\times}$, the character $\chi_{K,p}$ can be regarded 
as a character of $\Z_p^{\times}/N_{K/\Q}(O_p^{\times})$ and   
we have $\chi_K\in X(O)$. 
If $\chi$ is a genus character of $O$, then the value 
$\chi(a)$ for $a \in K_A^{\times}$ depends only on $\Q^{\times}N_{K/\Q}(a)$,
and (any power of) $\chi_K$ is trivial on the latter elements.
So genus characters of $O$ corresponds bijectively with 
\[
X(O)/\{\chi_K^{i};0\leq i\leq n-1\}, \qquad \text{ $n=[K:\Q]$.}
\]
For a genus character $\chi$ of $O$ corresponding to $\phi\in X(O)$
and $a=(a_p)\in K_A^{\times}$ with $a_p=(a_v)_{v|p}$ such that 
$N(a) \in \R_+^{\times}\prod_{p}\Z_p^{\times}$,
we have $\chi(a)=\phi(N(a))=\prod_{p}\phi_p(N(a_p))$ by definition,
But in general, $N(a)$ belongs to 
$\R_+^{\times}\prod_p\Z_p^{\times}$ only after 
multiplying an element of $\Q^{\times}$,
and in order to give exact values of $\chi(a)$,
we need this kind of adjustment. 
When $K$ is a quadratic 
field, we will describe $X(O)$ and the values of genus characters $\chi$ on 
ideals more precisely in the next section.

\section{Explicit genus characters 
for quadratic orders}
In the rest of the paper, we assume that $K$ is 
a quadratic extension of $\Q$ and denote the norm $N_{K/\Q}$ and the trace 
$Tr_{K/\Q}$ from $K$ to $\Q$ by $N$ and $Tr$, respectively.  
The notation $N$ is also used for the norm $N(\fa)$ of an ideal $\fa$
defined to be $[O_f:\fa]$, but we believe no confusion is likely to happen.
Assume that 
the maximal order of $K$ is written as 
$O_{max}=\Z+\Z\omega$. Then orders $O_f$ of 
$K$ correspond bijectively to positive 
integers $f$ called conductors by 
\[
O_f=\Z+\Z f\omega.
\]
We denote by $D_K$ the fundamental discriminant of $K$ and 
we say that $D=f^2D_K$ is the discriminant of $O_f$. 
We say that an ideal $\fa$ of $O_f$ is proper 
if 
\[
\{\alpha \in K:\fa\alpha \subset \fa\}=O_f.
\]
It is obvious that 
$\fa$ is proper if and only if $\fa_p$ is proper in 
$O_{f,p}=O_f\otimes_{\Z}\Z_p$ for all primes $p$, 
where the word proper is defined similarly for $O_{f,p}$.
Any principal ideal is obviously proper. 
So any locally principal ideal $\fa$ of $O_f$ is proper. 
Conversely we have 

\begin{lemma}\label{locallyprincipal}
Any proper fractional ideal of $O_f$ is locally principal.
\end{lemma}

\begin{proof}
Though this has been proved in \cite{ihara}, we 
give a shorter proof here. We may assume that $\fa$ is integral. 
For a proper integral ideal $\fa$, 
we have integers $a>0$, $\ell>0$, $d\in \Z$ 
such that 
\[
\fa=\ell(\Z a+\Z(d+f\omega))
\]
with $N(d+f\omega)=ac$ for some integer $c$ 
(\cite{arakawaibukaneko}).
It is enough to show that  
$\fa_p=\Z_p a+\Z_p(d+f\omega)$ is principal for any prime $p$.
If $p\nmid f$, then $O_{f,p}=O_{max,p}$ so 
the result is classical (even when $p$ splits). 
So we assume $p|f$. 
If $a \in \Z_p^{\times}$, then $\fa_p=O_{f,p}^{\times}$
so nothing to do. Next we assume $p|a$.  
Since 
\[
N(d+f\omega)=d^2+fdTr(\omega)+f^2N(\omega)=ac.
\]
and we assumed $p|a$, $p|f$, we have $p|d$.  
If $p|c$, then 
\[
(d+f\omega)(d+f\omega^{\sigma})/p=a(c/p)\in \fa  
\text{ for }\sigma\in Gal(K/\Q) \text{ with } \sigma\neq id.
\]
But since $f\omega=fTr(\omega)-f\omega^{\sigma}$,  
we have 
\[
(d+f\omega)/p=-(d+f\omega^{\sigma})/p+2(d/p)+(f/p)Tr(\omega)
\]
with $d/p$, $(f/p)Tr(\omega) \in \Z$, 
so we have $(d+f\omega)(d+f\omega)/p\in \fa$. 
On the other hand we have 
\[
a(d+f\omega)/p=(a/p)(d+f\omega)\in \fa. 
\]
So $\fa(d+f\omega)/p\subset \fa$. But 
$(d/p)+(f/p)\omega\not\in O_{f,p}$ so this contradicts the 
assumption that $\fa_p$ is proper. So we have $p\nmid c$.
This means $a\in (d+f\omega)O_p$, so $\fa_p=O_p(d+f\omega)$.
\end{proof}
The proper ideals are important classically since 
they correspond nicely to the binary quadratic forms (See \cite{arakawaibukaneko}). 

The principal genus $H(O)$ corresponds to square 
classes of locally principal ideals.
This can be seen as follows.
In Proposition \ref{iyanagatamagawa}, we may 
assume that $a\in b^{1-\sigma}K^{\times}U_+$, so 
$a$ is in the same class as $b^{1-\sigma}$ in the 
narrow sense. We regard $\Q_A^{\times}$ as a 
subset of $K_A^{\times}$ naturally
(i.e. for $v=\infty$ or rational prime, 
if $K_v=\Q_v+\Q_v$, then we embed 
$\Q_v$ diagonally and if $K_v$ is a field, we embed 
$\Q_v$ as a subfield.) Since 
\[
bb^{\sigma}\in \Q_A^{\times}=\Q^{\times}\R_+^{\times}
\prod_{p}\Z_p^{\times}\subset K^{\times}U_{+}(O),
\]
we have $b^{1-\sigma}K^{\times}U_+(O)
=b^2K^{\times}U_+(O)$, so 
$a$ belongs to the square classes in the narrow sense.
Hence a genus is a coset of the subgroup of $Cl^+(O)$ 
consisting of square classes in the narrow sense, 
and genus characters are nothing 
but a character of $Cl^+(O)$ of order at most two.
We will describe these characters explicitly in this section.
First we describe components of $N_{K/\Q}(H(O))$ at primes. 
For the sake of completeness and for reader's convenience, 
we review easy known results 
concerning $O_{max}$ for a while.
By the local class field theory, if $p$ is unramified in 
$K$, then we have 
$N(O_{max,p}^{\times})=\Z_p^{\times}$.
The following lemma is well known and easy to see.
\begin{lemma}\label{maxnorm}
(i) When $p$ splits in $K$, we have 
$K_p=\Q_p\oplus \Q_p$ and 
\[
N(K_p^{\times})
=\{p^n:n\in \Z\}\times \Z_p^{\times}.
\]
(ii) When $p$ is unramified and remains prime in $K$, we have 
\[
N(K_p^{\times})=\{p^{2n};n\in \Z\}\times \Z_p^{\times}.
\]
(iii) If $p$ is odd and ramified in $K=\Q(\sqrt{pm})$ where 
$m$ is an integer such that $p\nmid m$, then 
\[
N(K_p^{\times})=\{(-pm)^n:n\in \Z_p\}
\times (\Z_p^{\times})^2.
\]
Here $(\Z_p^{\times})^2$ is defined to be the set of 
square elements of $\Z_p^{\times}$. 
\\
(iv) If $p=2$ is ramified in $K=\Q(\sqrt{m})$ 
for an integer $m$ with $2\nmid m$ (so $m\equiv 3 \bmod 4$), we have 
\[
N(K_2^{\times})
=
\left\{\begin{array}{ll}
(-2)^n \times (1+4\Z_2) & \text{ if $m\equiv 3\bmod 8$,}\\
2^n \times (1+4\Z_2) & \text{ if $m\equiv 7 \bmod 8$}.
\end{array}\right.
\]
(v) If $p=2$ is ramified in $K=\Q(\sqrt{2m})$ for an integer $m$
with $2\nmid m$, we have 
\[
N(K_2^{\times})=
\left\{\begin{array}{ll}
\{2^n:n\in \Z\}\times \{1+8\Z_2,-1+8\Z_2\}
& \text{ if $m\equiv 1 \bmod 8$}, \\
\{(-2)^n:n\in \Z\}\times 
\{1+8\Z_2,3+8\Z_2\}
& \text{ if $m\equiv 3 \bmod 8$},\\
\{6^n:n\in \Z\}\times \{1+8\Z_2,-1+8\Z_2\}
& \text{ if $m\equiv 5 \bmod 8$}, \\
\{2^n:n\in \Z\}\times \{1+8\Z_2,3+8\Z_2\}
& \text{ if $m\equiv 7 \bmod 8$}. 
\end{array}\right.
\]
\end{lemma}

So the non-trivial character $\chi_p$ 
of $\Z_p^{\times}/N(O_{max,p}^{\times})$ is 
given as follows. For (i) and (ii), we have 
$\chi_p=1$. For (iii), we have 
$\chi_p(a)=\left(\frac{a}{p}\right)$ (the quadratic residue symbol).
For (iv), $\chi_2(a)$ is 
$\chi_{-4}(a)=\left(\frac{-4}{a}\right)$.
For (v), if $m\equiv 1 \bmod 4$, then 
$\chi_2(a)$ is $\chi_8(a)=\left(\frac{2}{a}\right)$.
For $m\equiv 3 \bmod 4$, we have 
$\chi_2(a)$ is $\chi_{-8}(a)=\left(\frac{-8}{a}\right)$.
For each quadratic field $K/\Q$. 
the character $\chi$ of $\prod_p\Z_p^{\times}$ is 
defined by $\prod_{p}\chi_p$ by taking $\chi_p$ on 
$\Z_p^{\times}$ as above, and we can prolong this 
naturally to a character of $\Q_A^{\times}$ 
by using the direct product decomposition 
\eqref{direct} of $\Q_A^{\times}$.
This is nothing but the character $\chi_K$ corresponding to 
the quadratic extension $K$ over $\Q$. 

This character is also given in another way as explained below.
If a fundamental discriminant $\delta$ of some quadratic field 
can be divided only by one prime, then we say $\delta$ is a prime
 discriminant. For example, for an odd prime $p$, if we write 
 $p^*=(-1)^{(p-1)/2}p$, then this is the unique prime discriminant 
 divided by $p$. For $p=2$, the prime discriminants are
 $-4$, $8$, $-8$.
For each prime discriminant divided by $p$, 
we define a Dirichlet character 
$\chi_{\delta}(a)=\left(\frac{\delta}{a}\right)$ as usual: 
We put $\chi_{\delta}(-1)=-1$ if $\delta<0$ and $=1$ if 
$\delta>0$. For a prime $p$ such that $p\nmid \delta$, 
we put $\chi_{\delta}(p)=1$ if $p$ splits in $\Q(\sqrt{\delta})$, 
$=-1$ if $p$ remains prime, and $=0$ if $p|\delta$.
For any integer $a=\epsilon q_1^{e_1}\cdots q_m^{e_m}$
with $\epsilon=\pm 1$ and primes $q_i$, 
we put $\chi_{\delta}(a)=\chi_{\delta}(\epsilon)
\prod_{i=1}^{m}\chi_{\delta}(q_i)^{e_i}$. 
Any fundamental discriminant $D_K$ of 
a quadratic field $K$ is uniquely decomposed 
into a product of prime discriminants $\delta_i$ 
as $D_K=\delta_1\cdots\delta_r$.
Then for any integer $a$, we define
\[
\chi_K(a)=\prod_{i=1}^{r}\chi_{\delta_i}(a):=\biggl(\frac{D_K}{a}
\biggr).
\]
In particular, we see that $\chi_K(-1)=1$ if $K$ is real and $-1$ if $K$ is
imaginary.
We may regard $\chi_K$ as a character 
$\prod_p\chi_{K,p}$ 
of $\prod_p\Z_p^{\times}$ where for each prime $p$, 
$\chi_{K,p}$ is the character of $\Z_p^{\times}$ already given 
just after Lemma \ref{maxnorm}.
The proof is as follows.
For an odd prime $p$ and an odd prime $q\neq p$,
by the quadratic reciprocity we have 
\[
\left(\frac{p^*}{q}\right)=\left(\frac{q}{p}\right).
\]
This is true even for $q=2$. Indeed 
for $p\equiv 1 \bmod 8$ and $p\equiv 5 \bmod p$, 
we have $\left(\frac{p}{2}\right)=1$ and $-1$, 
respectively. For $p\equiv 3 \bmod 8$ and $7\bmod 8$,
we have $\left(\frac{-p}{2}\right)=-1$ and $1$, 
respectively. We see that in all these cases, this is 
equal to $\left(\frac{2}{p}\right)$. When $p=2$, $\chi_{K,2}$ is the same as 
$\chi_2$ defined before.
We can also see easily that 
$\prod_p\chi_{K,p}(-1)=1$ for real $K$ and $=-1$ for 
imaginary $K$.
Of course the fact mentioned above are all  
classically well known.
We may prolong $\chi_K$ 
to the character of $\Q_A^{\times}$
trivial on $\Q^{\times}\times \R_+^{\times}$.
Then we have $Ker(\chi_K)=\Q^{\times}N(K_A^{\times})$. 
In this adelic setting, calculation of the value of $\chi_K$ at 
an element of $\Q_A^{\times}$ not in $\prod_p\Z_p^{\times}$ is 
easy. 
For example, for a prime $p$ with $p\nmid D_K$, 
put   
\[
[p]:=(1,\ldots,1,p,1,\ldots,1)\in \Q_A^{\times},  
\]
where $p$-component is $p$ and all the other 
components are $1$.
Then we have 
\[
\chi_K([p])=\chi_K((p^{-1},\ldots,p^{-1},1,p^{-1},
\ldots,p^{-1})=\prod_{q\neq p}\chi_{K,q}(p^{-1})=\left(\frac{D_K}{p}\right).
\]
Next we consider $N(O_{f,p}^{\times})$ for $p|f$.
We write $D=f^2D_K$.
We denote by $\ord_p(f)$ the $p$-adic order of $f$.

\begin{lemma}\label{fnorm}
Assume that $p|f$. \\
(1) If $p$ is odd, then 
\[
N(O_{f,p}^{\times})=(\Z_p^{\times})^2.
\]
(2) If $p=2$, then we have 
\[
N(O_{f,2}^{\times})
=
\left\{
\begin{array}{lll}
(i) & \Z_2^{\times} & \text{ if $ord_2(f)=1$ and $D_K$ is odd}, \\
(ii) & 1+4\Z_2 & \text{ if $ord_2(f)=1$ and $D_K
\equiv 12 \bmod 16$},
\\
(iii) & 1+4\Z_2 & 
\text{ if $ord_2(f)=2$ and $D_K\equiv 1 \bmod 4$},\\
(iv) & 1+8\Z_2 & \text{ if $D\equiv 0 \bmod 32$}. \\
\end{array}
\right.
\]
\end{lemma}
The above cases exhaust all the cases, since 
the case (iv) is whether $ord_2(f)=1$ and
$D_K\equiv 0 \bmod 8$, 
$ord_2(f)=2$ and $D_K\equiv 0 \bmod 4$, or $3\leq ord_2(f)$. 

\begin{proof}
If $p\neq 2$, then 
\[
N(x+yf\omega)=x^2+xyfTr(\omega)+f^2N(\omega)
\equiv x^2 \bmod p.
\]
If we put $y=0$, we see 
$(\Z_p^{\times})^2\subset N(O_{p,f}^{\times})$, 
so we have (1) by Hensel's lemma.
Now assume $p=2$. 
First of all, we note that 
$1+8\Z_2=(\Z_2^{\times})^2\subset N(O_{f,2}^{\times})$.
Since $\Z_2^{\times}/(1+8\Z_2) \cong 
\Z/2\Z\times \Z/2\Z$, we must see how many cosets 
of $1+8\Z_2$ appears. In case (i), we have 
$O_{f,2}=\Z_2+\Z_2\sqrt{D_K}$, so 
$N(x+y\sqrt{D_K})=x^2-y^2D_K$. So $N(O_{f,2}^{\times})$ contains
$-D_K$, $2^2-D_K=4-D_K$. Since 
$D_K\equiv 1 \bmod 8$ or $5 \bmod 8$, 
these generate $\{1,-1,3,-3\}\subset 
N(O_{f,2}^{\times})$, so we have 
$N(O_{f,2}^{\times})=\Z_2^{\times}$. In case (ii), 
we have $O_{f,2}=\Z_2+\Z_2 2\sqrt{m}$
and $N(x+2y\sqrt{m})=x^2-4my^2$.
If this belongs to $\Z_2^{\times}$, then $x$ should be odd. 
So the norm is $\equiv 1 -4y^2m$. 
This gives $1 \bmod 8$ for even $y$ 
and $5 \bmod 8$ for odd $y$ since $m\equiv 3 \bmod 4$, 
so (ii) is proved.
In the case (iii), we have $O_{f,2}=\Z_2+2\sqrt{D_K}$ and 
$x^2-4D_Ky^2 \equiv 1$ or $5 \bmod 8$. 
For (iv), elements of $O_{f,2}$ is written as 
$\Z_2+\Z_2 \sqrt{D}/2$  so
$N(x+y\sqrt{D}/2)=x^2-y^2(D/4)\equiv 1 \bmod 8$, 
so we prove (iv).
\end{proof}

In Lemma \ref{fnorm}, the corresponding non-trivial character of 
$\Z_p^{\times}/N(O_{f,p}^{\times})$ is 
$\left(\dfrac{p^*}{a}\right)$ for (1), trivial for (2)(i), 
$\chi_{-4}$ for (2)(ii) and (iii), and $\chi_{-4}$, $\chi_8$, 
$\chi_{-8}$ for (2)(iv). 
For a fixed discriminant $D=f^2D_K$, 
a character defined as a product of several local characters 
of $\Z_p^{\times}/N(O_{f,p}^{\times})$ appearing in 
Lemma \ref{maxnorm} and \ref{fnorm} 
is equal to a character $\chi_{\delta}$ corresponding to some 
fundamental discriminant $\delta$ of 
a divisor of $D=f^2D_K$ such that
$D/\delta\equiv 0$ or $1 \bmod 4$. 

Any divisor $\delta$ of $D=f^2D_K$ 
which is a fundamental discriminant of some quadratic field or $1$ such that  
$D/\delta\equiv 0$ or $1 \bmod 4$ is 
called a fundamental divisor of $D$
(Stammteiler in Weber \cite{weber}). For example, $1$ and $D_K$ are 
always fundamental divisors. 
For a fundamental divisor $\delta_1$ of $D$, 
there exists another fundamental divisor $\delta_2$ of $D$  
such that $\delta_1\delta_2=f_1^2D_K$
for some $f_1|f$. Or equivalently we may say 
$D=\delta_1\delta_2f_0^2$ for $f_0$ with $f_0f_1=f$. 
We say that such $\delta_1$ and $\delta_2$ 
are reciprocal. Here $\delta_2$ is determined uniquely by $\delta_1$. 
We have $\chi_{\delta_1}\chi_{\delta_2}=\chi_K$ 
(regarding $\chi_{\delta}$ as the trivial character when 
$\delta=1$, and taking the product so that the result becomes 
a primitive character, i.e. regarding the square of the same prime discriminant 
part as a trivial character.)
 
\begin{prop}[Weber \cite{weber}]
The set of reciprocal pairs of fundamental discriminants of 
$D=f^2D_K$ corresponds bijectively to the set of  
genus characters. In particular, if we denote by $\nu$ the number of 
odd divisors of $D$, then the genus number $g$ of $O$ in the narrow sense
is given as follows.
\[
g=\left\{\begin{array}{ll}
2^{\nu-1} & \text{ if $D\equiv 1 \bmod 4$ or $D\equiv 4 \bmod 16$}, \\
2^{\nu} & \text{ if $D\equiv 8,12,16,24,28 \bmod 32$}, \\
2^{\nu+1} & \text{ if $D\equiv 0 \bmod 32$}.
\end{array}\right.
\]
\end{prop}

\begin{proof}
Since we have 
\[
K_A^{\times}/H(U_+)\cong 
\Q^{\times}N(K_A^{\times})/\Q^{\times}N(U_+(O)),
\]
and $\chi_K(a)=1$ for $a\in \Q^{\times}N(K_A^{\times})$, 
the first part of the 
above proposition is obvious. The assertion on the genus number is obtained 
by a careful check of Lemma \ref{maxnorm} and \ref{fnorm}.
\end{proof}

A class $C \in Cl^+(O)$ in the narrow sense is said to be an ambig class 
if $C^{\sigma}=C$. Equivalently, this is a class $C$ satisfying $C^2=1$. 
For such a class, we can show 
$C$ contains a proper ideal $\fa$ such that 
$\fa=\fa^{\sigma}$. This is called an ambig ideal.
Traditionally, the genus number is obtained by 
counting ambig ideals up to equivalence.
For such proofs, see for example (the Japanese version of) 
\cite{arakawaibukaneko}.
(By the way, note that even if $C$ is of order two in the wide sense, $C$ might not 
contain an ambig ideal.)

Since genus characters are characters of 
$Cl^+(O_f)$, it is preferable to write it as a function on 
proper $O_f$ ideals. We explain this below.
A proper integral ideal $\fa$ of $O_f$ is said to be prime to $f$ if we have 
\[
\fa+fO_f=O_f.
\]
If $\fa$ is prime to $f$, then $\fa$ is a
proper ideal. This is equivalent to the condition that $N(\fa)$ is prime to $f$.
We denote by $I(O_f,f)$ the set of proper $O_f$ ideals 
prime to $f$. 
It is well known that we have a 
bijective multiplicative mapping from 
$I(O_f,f)$ to $I(O_{max},f)$ by 
\[
\fa\rightarrow \fa O_{max},
\]
preserving norm and products 
(See \cite{arakawaibukaneko}).  
So any ideal in $I(O_f,f)$ is uniquely 
decomposed into a product of prime ideals.
If $\fa$ is a proper ideal of $O_f$ not prime to $f$, then 
there exists $\alpha \in K_+^{\times}$ such that 
$\fa\alpha$ is prime to $f$ 
(easily proved by the weak approximation theorem
that claims $K$ is dense in $\prod_{v\in S}K_v^{\times}$ 
for any finite set of places of $K$, or see   
\cite{arakawaibukaneko} for a global
 proof), 
so to give values of genus characters at ideals, 
it is enough to consider values at prime ideals 
$\fp$ in $I(O_f,f)$.
(By the way, considering by ideles, it is clear that 
any proper ideal of $O_f$ not necessarily prime to $f$ is 
also decomposed uniquely 
to the product of ideals of $O_{f,p}$ whose norms are 
powers of $p$. But maximal $O_f$ ideals are not proper in 
general and there is no proper ideal of norm $p$ 
for $p|f$. In particular, there is 
no prime ideal decomposition for ideals of $O_{f,p}$
in general.)  

We give a formula below how to calculate $\chi(\fp)$ for prime ideals $\fp\in I(O_f,f)$
for a genus character $\chi$. 

\begin{theorem}\label{charactervalue}
Let $\delta_1$, $\delta_2$ be a reciprocal pair of 
fundamental divisors and $\chi$ be a genus character 
associated with the pair.
Then for a prime ideal $\fp\in I(O_f,f)$,
we have the following formula.
\\
(1) If $\fp$ is unramified in $K$, then 
we have 
\[
\chi(\fp)=\chi_{\delta_1}(N(\fp))=\chi_{\delta_2}(N(\fp)).
\]
(2) If $\fp$ is ramified, then $N(\fp)$ is prime to 
one of $\delta_i$ (say $\delta_1$).
Then we have 
\[
\chi(\fp)=\chi_{\delta_1}(N(\fp)).
\]
\end{theorem}

\begin{proof}
For $p\nmid f$, 
the prime ideal $\fp$ over $p$ corresponds with an idele
\[
a=(1,\ldots,1,a_p,1,\ldots,1)\in K_A^{\times}
\]
where $a_p\in K_p^{\times}$ is the $p$-th component such that 
$\fp_p=a_pO_{f,p}=a_pO_{max,p}$.
The value of a genus character $\chi$ on $a$ is determined by 
the value of $N(a)\in \Q_A^{\times}$ for the 
corresponding characters $(\chi_{\delta_1},\chi_{\delta_2})$. 
When $p$ remains prime in $K$, then we have 
$N(\fp)=p^2$ and 
$a_p=p\epsilon$ with $\epsilon \in O_{max,p}^{\times}$.
So we have 
\[
N_{K/\Q}(a)=(1,\ldots,1, p^2N(\epsilon),1,\ldots,1)\in \Q_A^{\times}.
\]
We must change this to an element of 
$\prod_{q}\Z_q^{\times}$ by multiplying an element of 
$\Q^{\times}$ to evaluate by the character of 
$\prod_q\Z_q^{\times}$. So we consider 
\[
p^{-2}N(a)=(p^{-2},\ldots,p^{-2},N(\epsilon),
p^{-2}\ldots,p^{-2}).
\]
Since any local character is of order two at $q\neq p$ 
and trivial at $p$  (since $\Z_p^{\times}=N(O_{p,f}^{\times})$), 
we have $\chi(\fp)=1$. We may write this as 
\[
\chi_{\delta_i}(N(\fp))=\chi_{\delta_i}(p^2)=1.
\]
If $\fp$ over $p$ is unramified and split in $K$, 
then we have $O_{f,p}=O_{max,p}=\Z_p\oplus \Z_p$ and 
$a_p=(p,1)$ or $(1,p)$. So $N(a_p)=p$. 
Multiplying $p^{-1}$ to $a$, we have 
\[
p^{-1}N(a)=(p^{-1},\ldots,p^{-1},1,p^{-1}\ldots,p^{-1}).
\]
So $\chi(\fa)=\prod_{q\neq p}\chi_{\delta_1,q}(p^{-1})=\chi_{\delta_1}(p^{-1})$. 
Since $\chi_K(p)=1$, this is of course equal to $\chi_{\delta_2}(p^{-1})$.
Since $\chi_{\delta_i}$ is of order 
two, this is equal to $\chi_{\delta_i}(p)$. 
So we have $\chi(\fp)=\chi_{\delta_1}(N(\fp))=\chi_{\delta_2}(N(\fp))$.
If $(p)=\fp\fp^{\sigma}$, then we also have $\chi(\fp^{\sigma})=\chi(\fp)$.

Finally, assume that $\fp$ is ramified. 
Then we have 
\[
N(a)=(1,\ldots,1,pu,1,\ldots,1)
\]
for some $u\in \Z_p^{\times}$. We have 
\[
p^{-1}N(a)=(p^{-1},\ldots,p^{-1},u,p^{-1},\ldots,p^{-1}).
\]
Since we assumed $p\nmid f$, one of $\delta_i$ does not contain $p$ as a factor.
Take such $i$ (say $i=1$). Then we have $\chi_{\delta_1,p}=1$ so 
$\chi_{\delta_{1,p}}(u)=1$ and  
\[
\chi(\fp)=\chi_{\delta_1}(p^{-1})=\chi_{\delta_1}(p).
\]
So the proof is completed.
\end{proof}

\begin{remark}
(1) When $p$ is ramified in the above proof,
if we take $\chi_{\delta_2}$ with $p|\delta_2$ instead, 
then we should have 
$\chi(\fa)=\chi_{\delta_2,p}(u)\prod_{q\neq p}\chi_{\delta_2,q}(p)$.
This is the same as $\chi_{\delta_1}(p)$ since 
$\chi_K$ is the character that is trivial on $\Q^{\times}
N(K_A^{\times})$. This can be also proved directly by 
using Lemma \ref{maxnorm}.\\
(2)  
Even when $\fa$ is not prime to $f$, we can also give 
some formula for $\chi(\fa)$ by the same sort of 
consideration, but in this case, we cannot describe it 
only by $N(\fa)$ since a value of some unit part like $u$ 
for the ramified case remains. 
We do not go into details since it seems we cannot avoid 
a bit complicated case studies. 
\end{remark}

\section{The $L$-functions of genus characters}
For an order $O_f$ of a quadratic field $K/\Q$, 
we fix a genus character $\chi$.
We define a $L$-function of proper ideals of $O_f$ with 
character $\chi$ as 
\[
L(s,O_f,\chi)=\sum_{\begin{subarray}{c}proper\ \fa\subset O_f\end{subarray}}
\dfrac{\chi(\fa)}{N(\fa)^s},
\]
 where the sum is taken over all integral proper ideals of 
 $O_f$ including those not prime to $f$. 
Since any proper ideal $\fa$ of $O_f$ is identified 
with a representative $a=(a_v)$ of $K_A^{\times}/K^{\times}U_{\infty,+}
\prod_{p}O_{f,p}^{\times}$
we have the unique decomposition of $\fa$ 
to the product  
$\fa=\prod_{p}(\fa_p\cap K)$
where $\fa_p=a_pO_{f,p}$, $a_p=(a_v)_{v|p}$.  
Here $N(\fa_p)=N(\fa_p\cap K)$ 
is a power of $p$. So it is clear that 
$L(s,O_f,\chi)$ is a product of the Euler
$p$-factors.

The Euler $p$ factors such that $p$ is prime to $f$ is simple.
We see this part first. Put  
\[
L_f(s,O_f,\chi)=\sum_{\fa\in I(O_f,f)}\frac{\chi(\fa)}
{N(\fa)^s},
\]
and we will write a formula for this. 
Assume that $\chi$ corresponds with 
a reciprocal pair of fundamental divisors 
$(\delta_1,\delta_2)$.
To simplify notation, we denote 
the Dirichlet character $\chi_{\delta_1}$, 
$\chi_{\delta_2}$ corresponding 
to $\delta_1$ and $\delta_2$ by 
$\phi$ and $\psi$. 
We will see the Euler $p$-factor for a prime $p$ with $p\nmid f$. 
If $p$ splits in $K$,  
then we have $pO_{f,p}=\fp_1\fp_2$ for some prime ideals $\fp_1$, $\fp_2$.
Then by Theorem \ref{charactervalue}, 
we have $\chi(\fp_1)=\chi(\fp_2)=
\phi(p)=\psi(p)$.  
So the Euler $p$ factor should be 
\[
\frac{1}{(1-\phi(p)p^{-s})(1-\psi(p)p^{-s})}.
\]
If $p$ remain prime, then $\chi(\fp)=\chi(pO_{f,p})
=1$, so the Euler factor is 
\[
\frac{1}{1-N(\fp)^{-s}}=\frac{1}{1-p^{-2s}}.
\]
But since $\chi_K(p)=-1$, we have 
$\phi(p)\psi(p)=-1$, so 
the Euler factor can be also written as 
\[
\frac{1}{(1-\phi(p)p^{-s})(1-\psi(p)p^{-s})}.
\]
If $p$ is ramified in $K$, then $\chi_K(p)=0$. 
By definition of the character $\chi$ and the assumption 
$p\nmid f$, we have $\phi(p)=0$ and $\psi(p)\neq 0$, or $\phi(p)\neq 0$ 
and $\psi(p)\neq 0$.
So again the Euler $p$ factor is 
\[
\frac{1}{(1-\phi(p)p^{-s})(1-\psi(p)p^{-s})}.
\]
So if we write $L_f(s,\phi)$ and $L_f(s,\psi)$ 
the usual Dirichlet $L$-functions omitting the Euler
$p$ factors such that $p|f$,  
we have 
\[
L_f(s,O_{f,p},\chi)=L_f(s,\phi)L_f(s,\psi).
\]
This is well known for $f=1$ (See Siegel \cite{siegel} for example).
The remaining problem is to determine the Euler $p$ factors 
for $p|f$. 
To determine these part, we need more precise description of 
proper ideals of $O_{f,p}$ for $p|f$. 
So we fix a prime $p$ with $p|f$.
For any integer $b\geq 0$, we write 
\[
R_b=\Z_p+\Z_pp^b\omega
\]
where $1$, $\omega$ is a basis over $\Z$ of $O_{max}$.
This is equal to $O_{f,p}$ for any $f$ with $ord_p(f)=b$. 
For a proper integral ideal $\fa=\alpha R_e$ of $R_e$,
if $N(\alpha)=p^du$ ($u\in \Z_p^{\times}$),
then we have $N(\fa)=N(\fa\cap K)=p^d$.

\begin{lemma}\label{properideal}
(1) If $\fa$ is a proper integral ideal of $R_e$ such that 
$ord_pN(\fa)\leq 2e$, then $ord_pN(\fa)$ is even. 
\\
(2) Proper integral ideals $\fa$ of $R_e$ such that 
$N(\fa)=p^{2c}$ with $c\leq e$ are given by 
$p^c\epsilon R_e$  for some 
$\epsilon\in R_{e-c}^{\times}$.
The number of such ideals of $R_e$ is equal to 
$[R_{e-c}^{\times};R_e^{\times}]$. \\
(3) Proper integral ideals $\fa$ of $R_e$ such that 
$N(\fa)=p^{2e+c}$ with $c\geq 0$ are given by  
$\fa=p^e\fa_0R_e$ for integral ideals $\fa_0$ 
of $R_0=O_{max,p}$ with $N(\fa_0)=p^c$. 
The number of such ideals of $R_e$ is equal to 
$[R_0^{\times}:R_e^{\times}]$ times the number of 
ideals $\fa_0$ of $R_0$ with $N(\fa_0)=p^c$.
\end{lemma}

\begin{proof}
Since proper ideals are locally principal, 
we write $\fa=(x+yp^e\omega)R_e$. 
Then $N(\fa)=p^c$ is equivalent to 
$ord_pN(x+yp^e\omega)=c$. 
We put $a=\ord_p(x)$. 
If we assume that $2e\leq c$, then $e\leq a$. 
Indeed if $a<e$, then 
\[
N(x+yp^e\omega)=x^2+xp^eyTr(\omega)+p^{2e}
N(\omega)\equiv 0 \bmod p^c,
\]
and 
\[
\\ord_p(x^2)=2a<a+e\leq \\ord_p(xp^eyTr(\omega))
\]
so $\ord_p(N(x+p^ey\omega))=2a<2e\leq c$, which is 
a contradiction. So we have $e\leq a$ and 
\[
\fa=p^e\alpha_0R_e \qquad \text{ for some }
\alpha_0=x_0+y_0\omega\in R_0.
\]
If we put $\fa_0=\alpha_0R_0$, then 
this is of course an ideal of $R_0$.
Here the generators of $\fa_0$ are written as 
$\alpha_0\epsilon$ with $\epsilon \in R_0^{\times}$, 
but generators of $\alpha_0\epsilon R_e$ are 
$\alpha_0\epsilon\epsilon_0$ with $\epsilon_0 \in R_e^{\times}$.
So, for each $R_0$ ideal $\fa_0$, the number of 
ideals $\fb_0$ of $R_e$ such that $\fb_0R_0=\fa_0$ is $[R_0^{\times}:R_e^{\times}]$,
Hence we prove (3).
(Note here that $\fa_0R_e$ is not an integral ideal of 
$R_e$ in general, but $p^e\fa_0R_e$ is.)
Next we assume that $c<2e$. We show that 
we have $c/2\leq a$. Indeed, if $a<c/2$, then 
$\ord_p(x^2)<c$, 
$c\leq [c/2]+e\leq \ord_p(xp^eyTr(\omega))$,
and $c< \ord_p(p^{2e}y^2N(\omega))$, so 
we have $\ord_p(N(x+yp^e\omega))=\ord_p(x^2)<c$ 
which is a contradiction. So we have $c/2\leq a$.
If $c$ is odd, then $(c+1)/2\leq a$, $e$, and we see that 
$\fa=p^{(c+1)/2}(x_0+p^{e-(c+1)/2}y\omega)$, so 
$c+1\leq ord_pN(\fa)$, which is a contradiction.
So there exists no ideal such that $ord_pN(\fa)$ is odd 
and $<2e$. So we have (1).  If $c$ is even, 
we rewrite $c$ by $2c$. Then we have 
\[
\fa=p^{c}(x_0+y_0p^{e-c}\omega)R_e.
\]
Since $N(\fa)=p^{2c}$,  we have 
$x_0+y_0p^{e-c}\in R_{e-c}^{\times}$. 
So the number of ideals is exactly equal to 
$[R_{e-c}^{\times}:R_e^{\times}]$.
\end{proof}

In order to count the number of ideals, we give necessary indices.
\begin{lemma}\label{count}
For $1\leq e-c$ we have 
\[
[R_{e-c}^{\times};R_e^{\times}]=p^c.
\]
For $e=c$ and $R_{e-c}=R_0$,  we have 
\[
[R_0^{\times}:R_e^{\times}]=p^{e-1}(p-\chi_K(p)).
\]
\end{lemma}

\begin{proof}
First we assume that $p$ splits in $K$. Then we have 
$K_p=\Q_p\oplus \Q_p$ and 
$R_0=O_{max,p}=\Z_p\oplus \Z_p$.
Since $\omega\in K$ is embedded in $\Z_p\oplus \Z_p$
by $\omega \rightarrow (\omega,\omega^{\sigma})$ 
where $\sigma$ is the non trivial automorphism of 
$K/\Q$, it is easy to see that 
\[
R_e=O_{f,p}=\{(a,b)\in \Z_p\oplus \Z_p; a\equiv b \bmod p^e\}.
\]
So if $e-c>0$, we have 
\[
R_{e-c}^{\times}/R_e^{\times}\cong 
(1+p^{e-c}\Z_p)/(1+p^e\Z_p)\cong \Z_p/p^{c}\Z_p
\]
and the order is $p^c$. If $e=c$, then 
\[
R_0^{\times}/R_e^{\times}=\Z_p^{\times}/(1+p^e\Z_p)
\]
and the order is $p^{e-1}(p-1)=p^{e-1}(p-\chi_K(p))$.
Next we assume that $p$ remains
prime in $K$. We write $P=pO_{max,p}$.
When $c=e$, the order of 
$R_0^{\times}/(1+P^e)\cong R_0^{\times}/(1+P)\times (1+P)/(1+P^e)$ is 
$(p^2-1)p^{2(e-1)}$. If $c<e$, then  
\[
R_{e-c}^{\times}/(1+P^e)\cong 
R_{e-c}^{\times}/(1+P^{e-c})\times 
(1+P^{e-c})/(1+P^e).
\]
and the order is $(p-1)p^{e+c-1}$.
So we have 
\[
[R_{e-c}^{\times}:R_e^{\times}]=\left\{
\begin{array}{ll}
p^c & \text{ if $c<e$}, \\
p^{e-1}(p+1) & \text{ if $c=e$}.
\end{array}\right.
\]
So noting that $\chi_K(p)=-1$, we have the assertion.
Finally we assume that $p$ ramifies in $K$  
and denote by $P$ the prime ideal of $O_{max,p}$. 
Then we have $P^2=pO_{max,p}$.
Assume that $c<e$. Then 
we have 
\[
[R_{e-c}^{\times}:1+P^{2e}]=[R_{e-c}^{\times}:1+P^{2(e-c)}]
[1+P^{2(e-c)}:1+P^{2e}].
\]
The order of this index is $(p-1)p^{e+c-1}$. 
So we have 
\[
[R_{e-c}^{\times}:R_e^{\times}]=p^c.
\]
If $c=e$, then 
\[
[R_0^{\times};R_e^{\times}]=[R_0^{\times}:1+P^{2e}]/
[R_e:1+P^{2e}]=(p-1)p^{2e-1}/(p-1)p^{e-1}=p^e.
\]
Since $\chi_K(p)=0$, we have the result.
\end{proof}

Finally we calculate the $L$-function of 
the genus character including the Euler $p$ factors with $p|f$.
We fix a genus character $\chi$ of $O_f$ corresponding to 
a pair $(\delta_1,\delta_2)$ of 
reciprocal fundamental divisors. 
By definition, we have $\delta_1\delta_2=f_1^2D_K$ of 
some divisor $f_1$ of $f$. 
If we put $f_0=f/f_1$, then we can also say that 
$D=\delta_1\delta_2f_0^2$. 
For any $p|f$, we write $m_p=\ord_p(f_0)=\ord_p(f/f_1)$. 
For any fundamental discriminant $\delta$ of 
a quadratic field $F/\Q$ and the Dirichlet 
character $\chi_{\delta}(a)=\left(\frac{\delta}{a}\right)$,
 we define the Dirichlet $L$-function $L(s,\chi_{\delta})$
 as usual by
 \[
 L(s,\chi_{\delta})=\prod_{p}
 (1-\chi_{\delta}(p)p^{-s})^{-1}.
 \]
 Here we regard $\chi_{\delta}(p)=0$ if $p|\delta$.
For the sake of simplicity, we write $\phi=\chi_{\delta_1}$ 
and $\psi=\chi_{\delta_2}$ as before. 

The following theorem was given in \cite{kanekomizuno}.
We give here a simple alternative proof.

\begin{theorem}\label{main}
Notation being as above, we have 
\begin{multline}\label{mainequation}
 L(s,O_f,\chi)
=  
L(s,\phi)L(s,\psi)\times
\\
 \prod_{p|f_0}
\frac{(1-\phi(p)p^{-s})(1-\psi(p)p^{-s})
-p^{m_p(1-2s)-1}(p^{1-s}-\phi(p))(p^{1-s}-\psi(p))}
{1-p^{1-2s}},
\end{multline}
where the product is taken over primes dividing the positive integer $f_0$ such that  
$D=\delta_1\delta_2f_0^2$. Here if $m_p=ord_p(f_0)=0$, 
then the $p$-factor of the product is regarded as $1$.
\end{theorem}

Before proving this, we prove
\begin{lemma}\label{vanish}
(1) Assume that $m_p<e$. Then we have 
$\phi(p)=\psi(p)=0$.
\\
(2) If $\phi(p)=\psi(p)=0$ and $p\nmid D_K$, then 
$m_p<e$.  
\end{lemma}

\begin{proof}
(1) Since we assumed $\ord_p(f_1)=e-\ord_p(f_0)=e-m_p>0$, 
at least one of $\delta_1$ and $\delta_2$ is divisible by $p$. 
Assume that $\ord_p\delta_1>0$ and 
$\ord_p(\delta_2)=0$. Then we have 
\[
\ord_p(\delta_1)=2(e-m_p)+\ord_p(D_K).
\]
Since $e-m_p>0$, we have $\ord_p(\delta_1)\geq 2$, and 
since $\delta_1$ is a fundamental discriminant, we 
should have $p=2$. So $ord_2(\delta_1)=2$ or $3$,  
and $ord_2(D_K)=0$ or $1$ for each case. 
The latter cannot happen, so 
we have $ord_2(D_K)=0$, $ord_2(f_1)=1$, and 
$\ord_2(\delta_1)=2$. Here $-4$ is a prime discriminant 
dividing $\delta_1$, so if we write 
$\delta_1=(-4)\delta_0$, then 
$\delta_0$ is an odd fundamental discriminant.
So we have 
\[
\delta_0\delta_2=-(f_1/2)^2D_K.
\]
Here since $f_1/2$ is odd, RHS is 
$\equiv 3 \bmod 4$. This contradicts that 
$\delta_0$ and $\delta_2$ are 
odd fundamental discriminants.
This means that any $p|f$ divides both 
$\delta_1$ and $\delta_2$ if $0<e-m_p$. So (1) is proved.
\\
(2) By definition we have 
$\chi_K=\phi\psi$. Here the product is taken 
so that the result is primitive.
So if $\phi_p=\psi_p$ with $p\nmid D_K$, 
we are taking $\phi_p\psi_p=1$. 
So it might happen that $\phi_p(p)=\psi(p)=0$ and 
$\chi_K(p)\neq 0$ in general. 
Now in our setting, 
if $\phi(p)=\psi(p)=0$, that is, if 
$1\leq \ord_p(\delta_i)$ for both $i=1$, $2$, then since 
\[
2\leq \ord_p(\delta_1)+\ord_p(\delta_2)=2 \ord_p(f_1)+\ord_p(D_K),
\]
and $\ord_p(D_K)=0$ by our assumption, we have $1\leq 
\ord_p(f_1)=e-m_p$, so $1\leq e-m_p$. 
\end{proof}

\begin{proof}[Proof of Theorem \ref{main}]
Any genus character $\chi$ of $O_f$ regarded 
as a character of $K_A^{\times}$ is trivial 
on $R_e^{\times}$ for $e=\ord_p(f)$.  
By the construction, a genus character 
$\chi$ of $O_f$ associated with a pair ($\delta_1$, $\delta_2$) such that 
$\delta_1\delta_2=f_1^2D_K$ 
can be regarded as a genus character of 
$O_{f/p^c}$ for $c\leq e$ 
if and only if $\ord_p(f_1)\leq e-c$. 
In other words, if we denote by $I_c$ the group of 
fractional ideals of $R_e$ defined by  
\[
I_c=\{\epsilon R_e: \epsilon \in R_{e-c}^{\times}\},
\]
then $\chi$ is trivial on $I_c$ if and only if 
$c\leq e-\ord_p(f_1)=m_p$.
So if $m_p<e$. we have 
\begin{equation}\label{vanishcond}
\sum_{\fa \in I_c/R_e^{\times}}\chi(\fa)N(p^c\fa)^{-s}
=0 \qquad \text{ for all $c=m_p+1,\ldots,e$}
\end{equation}
since $\chi$ is not trivial on $I_c$ in these cases
and $N(p^c\fa)=p^{2c}$ for any ideal $\fa\in I_c$. 

By definition we have $0\leq m_p\leq e$. 
First we assume that $m_p\neq e$. 
Then by Lemma \ref{properideal} (2) and \eqref{vanishcond}, we have no 
contribution for $L(s,O_f,\chi)$ from ideals $\fa$ with 
$N(\fa)=p^c$ with $m_p<c/2$. Indeed, if $c\leq 2e$ and $c$ is odd,
there is no such ideal by Lemma \ref{properideal} (1). If $c\leq 2e$ and 
$c=2c_0$ is even, then the ideals run over 
$p^{c_0}I_{c_0}$ and since $m_p<c_0\leq e$, the contribution vanishes by \eqref{vanishcond}.
If $c=2e+c_0$ with $0\leq c_0$ and $m_p< e$, then the ideals run over 
$p^e\fa_0 I_e$ for several ideals $\fa_0$ of $R_0$ with $N(\fa_0)=p^c$. 
and again 
the contribution is $0$ by \eqref{vanishcond} since $m_p<e$ means $ord_p(f_1)>0$ and 
$\chi$ is non-trivial on $I_e/R_e^{\times}$. 
In particular, if $m_p=0$, then by Lemma \ref{vanish}, we have 
$\phi(p)=\psi(p)=0$ unless $e=0$, and Theorem \ref{main} 
\eqref{mainequation} becomes $1$.
Next, consider ideals $\fa$ such that $N(\fa)=p^{2c}$ with $c\leq m_p<e$. Then 
$\fa$ runs over $p^cI_{c}$, and $\chi$ is trivial on 
these ideals. So by Lemma \ref{properideal} (2),
the contribution of such ideals to the $L$-function 
is given by $[R_{e-c}^{\times}:R_e^{\times}]p^{-2cs}$, 
and by Lemma \ref{count}, the total contribution from 
$c\leq m_p$ is given by   
\begin{equation}\label{first}
1+p^{1-2s}+p^{2(1-2s)}+\cdots+p^{m_p(1-2s)}
=
\frac{1-p^{(1+m_p)(1-2s)}}{1-p^{1-2s}}.
\end{equation}
By Lemma \ref{vanish}, we have $\phi(p)=\psi(p)=0$ in this case.
So the $p$ Euler factor of Theorem \ref{main} \eqref{mainequation} coincides 
with the above \eqref{first}.
Next we assume that $m_p=e$. This means that 
$\chi$ is regarded as a genus character of $O_{f/p^e}$.
Then the character $\chi$ is trivial on ideals $\fa$ of 
norm up to $p^{2e}$ and this part is given by \eqref{first}. 
If $\fa=p^e\fa_0R_e$ with 
ideals $\fa_0$ of $R_0$ then the value of the 
character $\chi(\fa)=\chi(\fa_0R_e)$ 
is the same value for
the corresponding character $\chi$ on $Cl^{+}(O_{f/p^e})$.
By using Lemma \ref{count}, we see that the contribution of this part is given by 
\begin{equation}\label{second}
\frac{p^{e-1}(p-\chi_K(p))p^{-2es}}
{(1-\phi(p)p^{-s})(1-\psi(p)p^{-s})}.
\end{equation}
So summing up \eqref{first} and \eqref{second},
and noting $\phi_K(p)=\phi(p)\psi(p)$ for any prime $p$ 
by Lemma \ref{vanish} (2), 
we obtain \eqref{mainequation} of 
Theorem \ref{main}.  
\end{proof}

\section{The genus number in the wide sense.}
If we replace the definition $U_+(O)$ by 
$U(O)=U_{\infty}\prod_pO_p^{\times}$ in section 2, where 
$U_{\infty}=(\R^{\times})^{r_1}(\C^{\times})^{r_2}$,
then the genus in the wide sense is defined by 
taking $a \in K_A^{\times}$ such that $N(a)\in \Q^{\times}N(U(O))$ 
in the same way. (Since we are assuming $K$ is cyclic 
over $\Q$, actually we have $r_1=0$ or $r_2=0$.)
The genus theory in the wide sense for maximal orders 
is given in general setting in Furuta \cite{furuta}, for example.
Maybe the concrete genus numbers in the wide 
sense for orders of a quadratic field are well known but 
we give the formula as an appendix as a continuation of 
the previous sections.
Of course this is nothing but the number of cosets in the ideal class group
in the wide sense over the group of square classes.
We will give an application 
of this formula in the next section. 

We first prove results for maximal orders for 
the sake of simplicity, and then 
state the result for general quadratic orders.

\begin{prop}\label{wide}
Let $K$ be a quadratic field.
Let $t$ be the number of prime divisors of the fundamental
discriminant $D_K$ of $K$.
\\
(1) If $K$ is imaginary, or if $-1 \in N(K^{\times})$, then 
the genus number in the wide sense  
is the same as the genus number in the narrow sense and given by 
$2^{t-1}$.
\\
(2) If $K$ is real and $-1\not\in N(K^{\times})$, then 
the genus number in the wide sense is 
$2^{t-2}$.  
\\
(3) For a real quadratic field $K$, we have 
$-1\in N(K^{\times})$ if and only if all the odd prime divisors
of $D_K$ are $1\bmod 4$.
\end{prop}

The condition (3) above is equivalent to the condition that 
the genus field of $K$ in the narrow sense (the abelian extension of $K$ 
corresponding to the principal genus classes in the narrow sense) 
is real, that is, unramified at infinite places.
(The genus field of $K$ in the wide sense is always real for real $K$
by the class field theory.)
Note also that by (3), we always have $t \geq 2$ in (2).

\begin{proof}
The results (1) and (2) are essentially due to \cite{furuta}, but we 
reprove it here.
The claim is clear when $K$ is imaginary, so we assume $K$ is real.
The difference from the narrow sense comes from 
$N(U_{\infty})=\R^{\times}$ since this is not contained 
in $\R_+^{\times}$. So correcting this part by multiplying
by $-1\in \Q$, the genus number in the wide sense is given by the index
\[
\frac{1}{2}
\bigg[\prod_{p}\Z_p^{\times}:\prod_{p}N(O_{max,p}^{\times})\bigcup 
(-1)\prod_pN(O_{max,p}^{\times})\bigg].
\]
This number is the same as the genus number in the 
narrow sense if and only if $-1\in N(O_{max,p}^{\times})$ 
for all 
$p$, and half of it if $-1\not\in N(O_{max,p}^{\times})$ for some 
$p$. For a real field, the condition that $-1\in N(O_{max,p}^{\times})$ for all 
primes $p$ is equivalent to the condition $-1\in N(K)$. 
Indeed, if $-1$ is a norm at all local places, 
then the global element $-1\in \Q^{\times}$ is a norm of an element of $K^{\times}$
by the Hasse norm theorem. Conversely, if 
$N(c)=-1$ for $c \in K^{\times}$, then $c\in O_{max,p}$ if $K_p$ is a field and 
$-1\in N(O_{max,p}^{\times})$. If $K_p=\Q_p\oplus \Q_p$, then 
$-1\in \Z_p^{\times}=N(\Z_p^{\times}\times \Z_p^{\times})$ always. So (1) and 
(2) are proved.
Now more concrete condition is as follows.
If $p$ does not ramify, then $-1\in \Z_p^{\times}=N(O_{max,p}^{\times})$ always. 
If $p$ ramifies, then by 
Lemma \ref{maxnorm}, for odd $p$, we have $-1\in 
N(O_{max,p}^{\times})$ 
if and only if $p\equiv 1 \bmod 4$. For $p=2$, by Lemma 
\ref{maxnorm} (iv) and (v), we have 
$-1\in N(O_{max,2}^{\times})$ if and only if $8|D_K$ and 
$D_K/8\equiv 1 \bmod 4$. But since we assumed 
$D_K>0$, the condition that all odd $p|D_K$ satisfy 
$p\equiv 1 \bmod 4$ means that $D_K/8\equiv 1\bmod 4$.
\end{proof}

Note that $-1\in N(K^{\times})$ is much weaker than the 
existence of a unit $\epsilon\in O_{max}^{\times}$ with $N(\epsilon)=-1$.
For example, for $K=\Q(\sqrt{221})=\Q(\sqrt{13\cdot 17})$,
we have 
$N\left(\dfrac{5+\sqrt{221}}{14}\right)=-1$ but 
the fundamental unit $(15+\sqrt{221})/2$ of $K$ has norm $+1$.
In this case, if we put 
\[
\fc=\Z 7+\Z\frac{5+\sqrt{221}}{2},
\]
then 
\[
\fc^2=\left(\frac{5+\sqrt{221}}{2}\right)O_K.
\]
Here we have $N(\frac{5+\sqrt{221}}{2})=-49$,
so square classes in the wide sense are equal to 
square classes in the narrow sense.
The genus numbers in the narrow sense and in the 
wide sense are both equal to $2$.
The genus field for $K$ is $\Q(\sqrt{13},\sqrt{17})$. 
For general quadratic $K$, if $N(c)=-1$ for $c \in K^{\times}$, then
by using the prime ideal decomposition of $cO_K$,  
we can easily see that $cO_K=\fc^{1-\sigma}$ for an 
ideal $\fc$ that is a product of prime ideals splitting in $K$. 
This means that $\fc^2=cN(\fc)O_K$ and that the square classes 
in the narrow sense and in the wide sense are the same.

Finally, for a quadratic order $O_f$ of general conductor $f$, 
we have the following results.

\begin{prop}
Put $D=f^2D_K$. 
The genus number for $O_f$ in the narrow sense is 
equal to the genus number in the wide sense if and only if 
the following two conditions are satisfied.
\\
(1) $p\equiv 1 \bmod 4$ for all odd $p|D$. \\
(2)  $D\not\equiv 0 \bmod 16$. \\
Otherwise, the genus number in the narrow sense is 
$2$ times the one in the wide sense.
\end{prop}

The proof is almost the same as the proof of 
Proposition \ref{wide} by using Lemma \ref{maxnorm} and 
\ref{fnorm}, so we omit it here. 

\section{Maximal orders of matrix algebras over algebraic number fields}
Let $F$ be an algebraic number field and $O_{max}=O_F$ be the ring of all integers of $F$.
A submodule $L$ of $F^n$ is said to be an $O_F$ lattice if 
it is finitely generated $O_F$ module and contains a basis of $F^n$.
A subring $\Lambda$ of $M_n(F)$ is said to be an $O_F$ order of $M_n(F)$ 
if it is an $O_F$ lattice in $M_n(F)$ and contains the unit matrix.
We denote by $F_+^{\times}$ the subgroup of elements of $F^{\times}$ which are positive under all real embeddings of $F$.
We define a subgroup $GL_n^+(F)$ of $GL_n(F)$ as 
\[
GL_n^+(F)=\{g\in GL_n(F):\det(g)\in F_+^{\times}\}.
\]
The number of $GL_n(F)$ conjugacy classes of maximal $O_F$ orders of $M_n(F)$
is called a type number of $M_n(F)$ 
(sometimes called in the wide sense in this paper). 
The similar number up to $GL_n^+(F)$ conjugacy classes will be called 
a type number in the narrow sense in this paper.
The purpose of this section is to characterize these numbers 
in terms of ideal classes. This has been known for type numbers in the wide sense
in \cite{helling} for $n=2$ and in \cite{arima} for general $n$.
We include this theory in the paper 
because when $n=2$ and $F$ is quadratic, these
two kinds of type numbers are given by 
the genus number in the wide sense and in the narrow sense, respectively.
The papers \cite{helling} and \cite{arima} use a global method but here 
we prove everything adelically. Most results below except for 
Propositions \ref{type} and \ref{typequadratic} are also found in \cite{reiner}.

We start from description of $O_F$ lattices $L\subset F^n$.
For any $g=(g_v) \in GL_n(F_A)$, we define $O_F^n g$ by 
\[
O_F^ng=\bigcap_{v<\infty}(O_{F,v}^ng_v\cap F^n).
\]
For any $O_F$ lattice $L$ and a finite place $v$ of $F$, 
we put $L_v=L\otimes_{O_F}O_{F,v}$. 
Then we have $L_v=O_v^n$ for almost all $v$ and it is clear that 
we have $L_v=O_v^ng_v$ for some $g_v\in GL_n(F_v)$ for all $v$. 
So any $O_F$ lattice is written as $O_F^ng$ for some $g \in G_A$.

For an $O_F$ lattice $L$, we write 
\[
\Lambda_L=\{g\in M_n(F); Lg\subset L\}
\]
and call it the right order of $L$.
Any maximal order $\Lambda$ of $M_n(F)$ is the right order of some $L$.
This is clear since $L\Lambda$ is again an $O_F$ lattice for any 
$O_F$ lattice $L$ and we have $\Lambda \subset \Lambda_{L\Lambda}$.
So any maximal order of $M_n(F)$ is written as 
\[
g^{-1}M_n(O_F)g:=\bigcap_{v<\infty}(g_v^{-1}M_n(O_{F,v})g_v\cap M_n(F))
\]
for some $g=(g_v) \in GL_n(F_A)$.

To write down global orbits of lattices and conjugacy classes
of maximal orders, we prepare adelic subgroups.

We denote by $GL_n^+(\R)$ the subgroup of elements of 
$GL_n(\R)$ with positive determinants.
We denote by $r_1$ and $r_2$ the number of real places and 
complex places of $F$. Put $U_{\infty}=GL_n(\R)^{r_1}\times GL_n(\C)^{r_2}$
and $U_{\infty,+}=GL_n^+(\R)^{r_1}\times GL_n(\C)^{r_2}$.
We put $U_0=\prod_{v<\infty}GL_n(O_{F,v})$ and 
$U=U_{\infty}U_0$ and $U_+=U_{\infty,+}U_0$. 
For ideal classes $C_i\in Cl(O_F)$ and $C_j^+ \in Cl^+(O_F)$, 
we fix representative ideles $a_i$ and $b_j$ in $K_A^{\times}$, respectively.
So we have  
\begin{align*} 
F_A^{\times}& =\bigsqcup_{i}a_iF^{\times}(\R^{\times})^{r_1}(\C^{\times})^{r_2}
\prod_{v<\infty}O_{F,v}^{\times} \quad (disjoint)\\
F_A^{\times}& = \bigsqcup_{j}b_jF^{\times}(\R_+^{\times})^{r_1}(\C^{\times})^{r_2}\prod_{v<\infty}O_{F,v}^{\times} \quad (disjoint).
\end{align*}
We define ideals $\fa_i$ and $\fb_i$ corresponding these by 
\[
\fa_i=\bigcap_{v<\infty}(a_{i,v}O_{F,v}\cap F),  
\quad 
\fb_j=\bigcap_{v<\infty}(b_{j,v}O_{F,v}\cap F).
\]
Here we may assume that infinite components 
$a_i$ and $b_i$ are all $1$.
We define diagonal matrices $g_i=diag(1,\ldots,1,a_i)$ and 
$h_j=diag(1,\ldots,1,b_j)$ in $GL_n(F_A)$. 
It is well known that we have the following double coset decomposition.
The proof is based on the strong approximation theorem on $SL_n$, and 
we omit the proof.
\begin{lemma}[\cite{kneser}]\label{class}
We have 
\begin{align*}
GL_n(F_A) & =\bigsqcup_i U g_i GL_n(F) \quad(disjoint),\\
GL_n(F_A) & = \bigsqcup_j U_+h_jGL_n(F) \quad (disjoint).
\end{align*}
\end{lemma}
If we denote by $GL_n^+(F_A)$ the subgroup of elements of $GL_n(F_A)$
whose infinite components are in $U_{\infty,+}$.
Then, since $F^{\times}$ contains an element with arbitrary sign at real places,
we also have 
\[
GL_n^+(F_A)=\bigsqcup_j U_+ h_jGL_n^+(F)
\quad (disjoint).
\]
It is easy to see that 
any $g\in GL_n(F_A)$ belongs to the double coset of $g_i$ if and only if 
$\det(g)$ belongs to the ideal class of $\fa_i$ and 
any $g \in GL_n^+(F_A)$ belongs to $h_j$ if and only if $\det(g)$ 
belongs to the ideal class of $\fb_j$.

In terms of global lattices, Lemma \ref{class} is written as follows.
\begin{lemma}\label{classglobal}
We fix an $O_F$ lattice $L$. \\
(1) There exists the unique ideal class $C_i\in Cl(O_F)$ such that 
the $GL_n(F)$ orbit of $L$ contains $(O_F,\ldots,O_F,\fa_i)$ for $\fa_i\in C_i$. \\ 
(2) There exists the unique ideal class $C_j^+\in Cl^+(O_F)$ such that 
the $GL_n^+(F)$ orbit of $L$ contains 
$(O_f,\ldots,O_F,\fb_j)$ for $\fb_j\in C_j^+$.
\end{lemma}
So any maximal order of $M_n(F)$ is $GL_n^+(F)$ conjugate to 
\[
\Lambda(\fa)=\begin{pmatrix} O_F & \cdots & O_F &\fa \\
\vdots & \cdots & \vdots & \vdots \\
O_F & \cdots & O_F & \fa \\
\fa^{-1} & \cdots & \fa^{-1} & O_F
\end{pmatrix}
\]
for some ideal $\fa$ of $O_F$. 
(This is well known up to $GL_n(F)$ conjugacy. See \cite{reiner} for example.)
We note that when $n=1$, we have $\Lambda(\fa)=O_F$ by definition.

Next problem is to describe dependence of $\Lambda(\fa)$ on $\fa$.
The relation $g^{-1}M_n(O)g=M_n(O)$ for $g=(g_v) 
\in GL_n(F_A)$ means 
that $M_n(O_{F,v})g_v$ $(v<\infty)$ is a two sided ideal of $M_n(O_{F,v})$.
It is well known that any two sided ideal is written as 
$c_vM_n(O_{F,v})$ for some $c_v\in F_v^{\times}$
(See \cite{reiner}).  
Now take $k_1$, $k_2\in GL_n^+(F_A)$ and assume that  
$g_0^{-1}k_1^{-1}M_n(O)k_1g_0=k_2^{-1}M_n(O)k_2$ for some 
$g_0\in GL_n^+(F)$. 
Assume that $k_1$ and $k_2$ belong to the double cosets of $h_i$ and $h_j$
in Lemma \ref{class}, respectively. Then by 
the above consideration, there exists 
$c\in F_A^{\times}$ whose infinite 
components are $1$ such that 
\[
U_+h_iGL_n^+(F)=U_+ch_jGL_n^+(F).
\]
This means that $\det(ch_j)=c^nb_j$ belongs to the narrow class of $b_i$.
The argument for conjugacy classes 
with respect to $GL_n(F)$ is similar. So we have
\begin{prop}\label{type}
(1) $\Lambda(\fa_i)$ and $\Lambda(\fa_j)$ are $GL_n(F)$ conjugate if and only if 
$\fa_i$ and $\fa_j\fc^n$ belong to the same ideal class in the wide sense
for some fractional ideal $\fc$ of $O_F$. \\
(2) $\Lambda(\fb_i)$ and $\Lambda(\fb_j)$ are $GL_n^+(F)$ conjugate 
if and only if $\fb_i$ and $\fb_j\fc^n$ belong to the same ideal class 
in the narrow sense for some fractional ideal $\fc$ of $O_F$.
\end{prop}

Of course the type numbers for both cases are the orders of 
ideal class groups $Cl(O_F)$ and $Cl^+(O_F)$ divided by 
$n$-th power classes, respectively. 
In particular, when $n=1$, then the type numbers 
are one. When $n=2$ and $F$ is quadratic over $\Q$,
these are genus numbers.  

\begin{prop}\label{typequadratic}
When $K$ is a quadratic field over $\Q$, 
the number of maximal orders of $M_2(K)$ up to $GL_2(K)$ conjugation
is equal to the genus number in the wide sense, and 
up to $GL_2^+(K)$ conjugation is equal to the genus number in the narrow sense
\end{prop}

Example: When $K=\Q(\sqrt{3})$ then 
the genus number in the wide sense is $1$, while the one in the narrow sense
is $2$. Representatives of maximal orders up to $GL_2^+(K)$ conjugacy classes 
are given by 
\[
M_2(O_K) \quad \text{and} \quad 
\begin{pmatrix} O_K & \sqrt{3}O_K \\
(\sqrt{3})^{-1}O_K & O_K \end{pmatrix}.
\]
These are conjugate by $\begin{pmatrix} 1 & 0 \\ 0 & \sqrt{3}\end{pmatrix}$
but not conjugate by any element of $GL_2^+(K)$.

\end{document}